\documentclass[10pt,a4paper]{article}
\usepackage{csquotes}
\usepackage{amsmath,amsfonts,amssymb,amsthm}
\usepackage{graphicx,verbatim,hyperref,etoolbox,svg}
\usepackage{longtable,multirow,booktabs}
\usepackage{caption, subfig}
\usepackage[table,xcdraw]{xcolor}
\usepackage{listings}

\usepackage[top=2cm, bottom=2cm, left=3cm, right=3cm]{geometry}

\usepackage[noblocks]{authblk}	

\makeatletter
\def\maketitle{{%
  \renewenvironment{tabular}[2][]
    {\begin{flushleft}}
    {\end{flushleft}}
  \AB@maketitle}}
\makeatother

\newtheorem{theorem}{Theorem}[section]

\newtheorem{lemma}[theorem]{Lemma}
\newtheorem{remark}[theorem]{Remark}
\newtheorem{definition}[theorem]{Definition}
\theoremstyle{definition}

\DeclareMathOperator*{\argmax}{arg\,max}
\DeclareMathOperator*{\argmin}{arg\,min}

\usepackage[greek,english]{babel}
\usepackage{newunicodechar}
\newunicodechar{ℝ}{\mathbb{R}}
\newunicodechar{α}{\alpha}
\newunicodechar{β}{\beta}
\newunicodechar{δ}{\delta}
\newunicodechar{Δ}{\Delta}
\newunicodechar{ε}{\epsilon}
\newunicodechar{θ}{\vartheta}
\newunicodechar{Θ}{\Theta}
\newunicodechar{ψ}{\psi}
\newunicodechar{λ}{\lambda}
\newunicodechar{Π}{\Pi}
\newunicodechar{σ}{\sigma}
\newunicodechar{Ω}{\Omega}

\newunicodechar{∈}{\in}
\newunicodechar{⇒}{\Rightarrow}
\newunicodechar{→}{\to}
\newunicodechar{•}{\bullet}
\newunicodechar{≤}{\le}
\newunicodechar{≥}{\ge}
\newunicodechar{∀}{\forall}
\newunicodechar{∞}{\infty}

\usepackage[backend=biber,style=numeric-comp,sorting=ynt]{biblatex}
\title{ForwardFlow: Simulation only statistical inference using deep learning}

\author{
	Stefan Böhringer$^*$\\
		Department of Biomedical Data Sciences\\
		Leiden University Medical Center\\
		Einthovenweg 20, 2333 ZC Leiden, The Netherlands\\
		\texttt{s.boehringer@lumc.nl} \\
		$^*$corresponding author\\
}

\begin{document}

\maketitle


\begin{abstract}
    Deep learning models are being used for the analysis of parametric statistical models based on simulation-only frameworks. Bayesian models using normalizing flows simulate data from a prior distribution and are composed of two deep neural networks: a summary network that learns a sufficient statistic for the parameter and a normalizing flow that conditional on the summary network can approximate the posterior distribution.
Here, we explore frequentist models that are based on a single summary network. During training, input of the network is a simulated data set based on a parameter and the loss function minimizes the mean-square error between learned summary and parameter. The network thereby solves the inverse problem of parameter estimation. We propose a branched network structure that contains collapsing layers that reduce a data set to summary statistics that are further mapped through fully connected layers to approximate the parameter estimate. We motivate our choice of network structure by theoretical considerations.

In simulations we demonstrate three desirable properties of parameter estimates: finite sample exactness, robustness to data contamination, and algorithm approximation. These properties are achieved offering the the network varying sample size, contaminated data, and data needing algorithmic reconstruction during the training phase. In our simulations an EM-algorithm for genetic data is automatically approximated by the network.

Simulation only approaches seem to offer practical advantages in complex modeling tasks where the simpler data simulation part is left to the researcher and the more complex problem of solving the inverse problem is left to the neural network. Challenging future work includes offering pre-trained models that can be used in a wide variety of applications.
\end{abstract}


{\bf Declaration of interest:} The authors declare no conflict of interests.\par

\newpage

\section{Introduction}

Simulation based approaches have a long tradition in statistical inference problems and have recently become more prominent\cite{cranmer_frontier_2020}. Approximate Bayesian Computation (ABC) is an early example and has been employed in complex inference problems where it is difficult or impossible to evaluate the data likelihood \cite{beaumont_approximate_2002}. More recently, normalizing flows have been established as an alternative for Bayesian inference \cite{papamakarios_normalizing_2021}. To approximate the posterior distribution, normalizing flows do not need to evaluate the likelihood \cite{radev_bayesflow_2020}. Normalizing flows employ deep neural networks (DNNs) to establish the normalizing mapping through a training process using simulations \cite{radev_bayesflow_2020}. After training, sampling from the posterior can make use of the computationally efficient use of DNN inference \cite{radev_bayesflow_2023}. In these examples, the data likelihood does not have to be evaluated which is an advantage in practical modeling. Both ABC and normalizing flows require sufficient summary statistics to properly construct the posterior distribution. For ABC, summary statistics have to be user supplied and might require deep insight into the analysis problem. Normalizing flows employ internal networks, summary networks, to automatically construct summary statistics. An additional restriction on the network structure arises for normalizing flows in that the mapping between the normal and posterior distribution has to be bijective. Another advantage of simulation based approaches is that they can implicitly account for finite sample properties by training networks on appropriate sample sizes \cite{radev_bayesflow_2023}. Arguably, this is an advantage over other approximations such as variational Bayes \cite{blei_variational_2017}.

In this paper, we investigate simulation only approaches which are based on a single network which learns sufficient summaries for the data. The focus is on offering a simulation-only approach for parametric models that can be implemented with little effort put into hyper-parameter tuning \cite{raiaan_systematic_2024}. A second aspect is robustness, which can be easily added by training the estimator using contaminated data. In contrast to {\it BayesFlow} with a more complex network structure, training is faster and the summary network can be potentially also be simplified. These networks directly allow to implement frequentist inference which is explored. In a second step, Bayesian models can be recovered from frequentist models using ABC methodology.

Some previous work has investigated statistical properties of deep neural networks. Results are either based in idealized assumptions ({\it e.g.} \cite{cybenko_approximation_1989}) or concerned with model evaluation only ({\it e.g.} \cite{goodfellow_deep_2016}). In general, such investigations face the difficulty that current DNNs are quantized at rather low numeric accuracy \cite{li_contemporary_2024}, a fact, that is difficult to account for in theoretic investigations. We here make some effort to motivate the choice of network structure by theoretical arguments.

The paper is structured as follows. In the second section the approach is described in terms of statistical models, general DNN structure, and some theoretical properties. The third section describes the network structure in detail, offering a family of models that can be tailored to the problem at hand. The next section describes simulations and is followed by a section describing results. In the final section, results are discussed and an outlook is given on future research.

\section{Methods}

For a parametric model, let $P_θ,\, θ \in \Theta \subset \mathbb{R}^k $ be the family of distributions.
Denote with $X: Ω → \mathcal{X}^N$, $X = (X_1, ..., X_N)$ a random variable, where $X_i = (X_{i1}, ..., X_{iK}) \sim P_{θ}$ is $K$-dimensional data sampled from this model.
Statistically, it is of interest to solve the inverse problem, {\it i.e.} find the mapping $\hat{θ}:\mathcal X^N → Θ$ that maps data to an estimate of the parameter.
For statistical inference, it is required to derive the distribution of $\hat{θ}(X)$. For frequentist problems, the distribution of interest is the confidence distribution \cite{xie_confidence_2013}, {\it i.e.} the formal distribution on $\Theta$ from which all-level confidence sets for $\theta$ can be derived. Formally, the confidence distribution is implicitly defined by a mapping $\mathrm{C}_{conf}: \mathcal X^N \to \Theta$ so that $P(C_{conf}(X, α) \ni \theta) ≥ α \, \forall α \in (0, 1) $, where $\mathcal X^N$ is the event space induced by $X$.\par

For a Bayesian model, the target distribution is the posterior $P(\theta | X)$ with prior distribution $\theta_0 \sim P_{pr}$, where data $X$ is given by $X_i \sim P_{θ_0}$ as above. The posterior can also be implicitly defined by $P(\theta_0 \in C_{post}(X, α, P_{pr})) ≥ α \, \forall α \in (0, 1) $.\par

{\it ForwardFlow} aims to approximate these distributions based on on estimator $\hat \theta$ which is learned by a deep neural network.

\subsection{Frequentist inference}

Denote with $\hat{θ}$ the estimator $\mathcal X^N → \Theta$, $\hat \theta := \hat \theta(X)$. We also assume that parameters $θ$ are drawn from a training distribution governed by a dispersion parameter $σ > 0$, so that $θ \sim P_{tr, σ}$. Without loss of generality, we assume $σ$ to be one-dimensional.
We can then define
\begin{eqnarray*}
    \hat{θ}_{σ} & := & \argmin_g E_{σ}(E_{θ}( (g(X) - θ)^2 )), \\
    \hat{θ} & := & \lim_{σ → \infty} \hat{θ}_{σ}
\end{eqnarray*}


where $g \in \mathcal{G}$ is a reasonably large class of functions. These definitions reflect the fact that $\hat{θ}$ is approximated using a neural network and training data has to be generated. The final estimator is the result for a ``large'' dispersion parameter $σ$ which, in practice, is chosen empirically. The class $\mathcal{G}$ is defined by the network topology and network capacity \cite{goodfellow_deep_2016}.

\begin{remark}
    Under standard regularity conditions for the maximum likelihood estimator to exist, $\hat{θ}$ exists.
\end{remark}

$\hat{θ}$ can be interpreted as the ML estimator for $θ$ as it estimates the posterior mode for the prior distribution $P_{tr,σ}$ under a Bayesian interpretation. Then, $\hat{θ}$ converges to the ML estimator as $P_{tr,σ}$ becomes uninformative as $σ → ∞$.

\begin{remark}
    If an efficient estimator exists, $\hat{θ}$ is efficient for $θ$
\end{remark}

As $\hat{θ}$ directly targets the $\mathrm{MSE}(\hat{\theta}) = \mathrm{Var}_{\theta}(\hat\theta)+ \mathrm{Bias}_{\theta}(\hat\theta,\theta)^2$, $\hat{θ}$ is bias free, if possible, as otherwise a bias-free $\hat{θ}'$ with corresponding estimate $θ'$ could be found with $\mathrm{MSE}(\hat{\theta'}) < \mathrm{MSE}(\hat{\theta})$

\subsubsection{Robust inference}

The estimator $\hat{θ}$ can be designed to be unbiased in the presence of data contamination. Informally, contaminations considered here are mappings
    $f_c: \mathcal{X} \times U → \mathcal{X}$ 
which transforms the data in an identifiable way. Here,
    $U = (U_1, ..., U_m), U_i \,iid \sim U(0, 1)$
introduces randomness into the transformation. An example would be a function replacing a value by a value drawn from an outlier distribution with a certain probability.

We assume that $\hat{θ}$ is unbiased, {\it i.e.} $E(\hat{θ}(\mathcal{X})) = θ$. Contamination is then formalized as follows.

\begin{definition}
    Under the assumptions of the previous paragraph, let $f_{cn}$ be a sequence of measurable functions
        $f_{cn}: \mathcal{X}_n \times U_n → \mathcal{X}_n$, $U_n \sim U(0, 1)^{N \times m}$
    multivariate iid uniform. $(f_{cn})_n$ is called a contamination sequence.
    The induced bias function $b_n$, by $(f_{cn})_n$, is defined as
    $$
        b_n(θ) := E_{θ,U}(\hat{θ} (f_{cn}(X_n, U_n)) - \hat{θ}(\mathcal{X_n}))
    $$
    The bias function is called consistent, if and only if $b_n$ does not depend on sample size, {\it i.e.} $b_n(θ) = b(θ)$. $θ_c = E_{θ,U}(\hat{θ} (f_{cn}(X_n, U_n)))$ is called the biased parameter.
\end{definition}

Sequences $f_{cn}$ can be constructed as product functions from a base contaminator $f_c$ acting on a single sample. We now characterize types of bias that can be handled by a neural network approach. 

\begin{definition}
    Let $(f_{cn})_n$ be a contamination sequence and $\hat{θ}$ be a consistent. $(f_{cn})_n$ is called a bijective contamination, if and only if $ θ_c → (θ, b_n(θ))$ is bijective, for the induced bias function $b_n(θ)$.
\end{definition}

Important examples of data contamination that are covered by the following lemmata are missing data and outliers. Data contamination is here defined by a mapping of the data that affects only part of the data. Here, we assume that the corresponding mapping is the product of one unit at a time mappings, {\it i.e.} contamination does not depend on multiple samples at a time. The missing at random (MAR) assumption would be one example falling into this class of contamination. In this case $f_{cn}$ could be represented by a logistic regression model determining the probability of contamination which is conditionally applied according to a uniform component of $U_n$.

\begin{remark}\label{rem:contamination}
    Let $f_n$ be a bijective contamination., $X_n' = f(X_n)$, $\hat \theta(X_n)$ consistent. Then there exists as de-biasing function $g_n: \mathcal{X}^n → Θ$ as follows:
    \begin{align*}
        \exists g_n: E\left(\hat \theta(X_n') + g_n(X_n') \right) = θ.
    \end{align*}
\end{remark}

$g_n$ is simply defined as the negative bias, {\it i.e.} $E(θ - \hat \theta(X_n'))$.

\begin{lemma}\label{lemma:contamination}
    Let $(f_{cn})_n$ be a contamination function. Assume that $f_{cn}$ acts only on a known part of the data, {\it i.e.} for $X'_n = f_n(X_n)$, $X'_{ni} = X_{ni}, 1 \leq i \leq k$. Then $(f_{cn})_n$ is bijective.
\end{lemma}

The proof is given in the appendix. While the existence of a de-biasing function is guaranteed by the previous remark, the lemma allows to easily check existence of the debiasing function in some situations. For example, missing data is covered by this case but not outliers.

The motivation for characterizing situations with bias is that the de-biasing function can be approximated by a neural network and as long as it exists, de-biasing is automatic when an appropriately trained model is used.

\subsubsection{Confidence distribution}

If a consistent parameter estimator $\hat{θ}$ exists, confidence intervals can be derived the confidence distribution which can be approximated by a bootstrap procedure. Bootstrap samples can be quickly generated and analyzed as a single batch by neural networks.

\subsection{Bayesian models}

The estimator $\hat{θ}$ allows to recover the confidence distribution but does not give access to the posterior distribution of a Bayesian model. We suggest to use Approximate Bayes Computation (ABC) which uses a filtering step on draws from the prior to approximate the posterior empirically.
For a sufficient statistic $T$, $P(\theta | X) = P(\theta | T(x))$. For a sample of statistics $T_i = T(X_i), X_i \sim P_{θ_i}, θ_i \sim P_{pr}, i = 1, ..., n$,

$$
    \hat{Θ} := \{ θ_i | d(T_i, T_{data}) < ε \}
$$
approximates the posterior, where $d$ is a metric on the space of summary statistics. By definition, $\hat{θ}$ is a minimal, sufficient statistic which is used for the ABC algorithm. The choice of $ε$ depends on the underlying model. Typically, $ε$ is chosen based on a quantile, such as the 1\% quantile of all distances which makes the algorithm inefficient.

\subsubsection{Importance Sampling}

To improve efficiency of ABC sampling, importance sampling can be used, {\it i.e.} samples can be drawn from a distribution preferentially targeting the posterior distribution. The importance sampling scheme suggested here assumes that an initial ABC sample has been drawn from the posterior  $\Theta_* = \{ \theta_i | \theta_i \,\mathrm{accepted} \}, N =|Θ_0|$. For ensuing draws, a new prior distribution is defined a mixture of normals centered at the accepted draws, {\it i.e.} $Z = \sum_i^N \frac{1}{N} Z_i$ with $Z_i \mathrm{iid} \sim N(\theta_i, \Sigma)$, 
    $\Sigma = \frac{s}{N - 1} Θ_0 Θ^T_0$, where $Θ_0$ is interpreted as a matrix of row vectors $θ_i$ and $s$ is a scale factor $s > 0$.

Using the modified prior, samples can be drawn according to $θ'_1 \sim Z$, $Θ_1 := \{ θ'_i | d(T_i, T_{data}) < ε \}$. This sample has a higher acceptance probability than the original sampling scheme, according to the following lemma:

\begin{lemma}\label{lemma:importance}
    Under posterior concentration, {\it i.e.} $E_{\mathcal X_n}(P_{\pi(\mathcal X_n)}( d(\theta, \theta_0) < \epsilon_n)) \to 1\,$ and known covariance matrix $\Sigma$
    $$
          P(d(\hat \theta^*, \hat \theta) < ε)
        ≥ P(d(\hat \theta^{'}, \hat \theta) < ε),
    $$
    for $n$ large enough, $d(\cdot, \cdot)$ metric on the parameter space, ${\mathcal X_n}$ the sampling space for sample size $n$.
\end{lemma}

Due to the modification, accepted samples $Θ_1$ can be related to the posterior distribution by the ratio $w_{iε}$, defined as follws:

$$
    P(d(\hat \theta^{'}, \hat \theta) < ε) =
        \frac{ P(d(\hat \theta^{'}, \hat \theta) < ε) }{P(d(\hat \theta^{'}, \hat \theta') < ε)} P(d(\hat \theta^{'}, \hat \theta') < ε)
        = w_{iε} P(d(\hat \theta^{'}, \hat \theta') < ε).
$$

$w_i$ can be defined as the ratio of the corresponding densities. In downstream analyses, such as density or quantile estimation, sample $Θ_1$ has therefore to be inversely weighted by $w_i$.

The scaling factor controls the concentration of samples in $Θ_1$ around the initial posterior sample and thereby controls the acceptance rate. Small $s$ lead to high acceptance rates whereas large values of $s$ lower acceptance rates. The exploration of the posterior space follows an inverse relationship with respect to $s$ such that the optimal choice of $s$ is subject to empirical investigation.

\subsection{Training distribution}

The deep network approximates the parameter estimator $\hat{θ}: \mathcal{X}^n → Θ$ which is a deterministic function. We assume that the network first finds a sufficient statistic ($\hat{θ}_S$ which is then mapped into the parameter space ($\hat{θ}_M$, such that
$\hat{θ} = \hat{θ}_M \circ \hat{θ}_S$.

Then, conditionally on $\hat{θ}_S(\mathcal X) = s$, $\hat{θ}$ finds the estimator by point-wise minimization of
$$
    \hat{θ}
    = \argmin_{θ|\hat{θ}_S} \{ (θ_i - θ)^2 | (\hat{θ}_i, θ_i), \hat{θ}_i = \hat{θ} \}
    = \bar{θ} = \frac{1}{N} \sum_i θ_i,
$$
{\it i.e.} among the training data sets with the same sufficient statistic $s$, the estimator is chosen as the sum-of-squares minimizer. Note that this decomposition implies that $\hat{θ}$ depends on the distribution of $θ$ as used in the training process. If the training distribution is interpreted as a prior, $\hat{θ}$ can be interpreted as the expected value of the posterior distribution. In order to allow for frequentist inference, the training distribution has to be uninformative, similar to the situation in objective Bayes analyses \cite{jeffreys_theory_1998}. As parameter values have to be drawn during training, a flat prior cannot be chosen, in general. In practice, a trade-off between training complexity and faithful frequentist inference has to be determined empirically.

%
%

\subsection{Network structure}

The parameter estimator $\hat{θ}: \mathcal{X}^n → Θ$ that is to be approximated by the network is a deterministic function. In principle, there is no uncertainty about the network weights up to non-identifiability so that after training, the estimator is a parametrized function of the network weights $β$. Due to over-parametrization, the network relies on random weight initialization to break symmetry and implicit regularization to find a local optimum using stochastic optimization \cite{kingma_adam_2017}. The goal is to design a generic network structure that can approximate a wide class of estimators.\par

Figure \ref{fig:network} shows an example of a branched network structure that is proposed in this paper. The branching is motivated by Roa-Blackwellization and finite sample size exactness, both being discussed below. For simplicity of exposition, we assume tabular i.i.d. data, {\it i.e.} input of the form $X \in \mathbb{R}^{N \times M}$. The networks accepts batches of data, {\it i.e.} a tensor $B \in \mathbb{R}^{B \times N \times M}$. For a neural network, implicit parallelization over the first index takes place. Independence of observations that also the second index should be handled in parallel which implies that coordinate-wise dense layers are used, {\it i.e.} the dense layer maps over the second dimension of the input tensor which are the individual samples for tabular data (see below). The general structure is that the input is fed into several branches with varying depth of coordinate-wise dense (or compatible) layers which end in custom so-called collapsing layers. These layers map tensors slices to a single number thereby reducing dimensionality of the tensor. Collapsing layers include the computation of mean values, covariances, or standard deviations. After collapsing several layers of dense mappings provide for further transformations. All remaining branches are concatenated and dense layers map the remaining single branch to the required vector size of the parameter vector which is the final output.

The coordinate-wise dense mapping is defined as follows for input
    $X \in \mathbb{R}^{N \times K}$
and output
$X' \in \mathbb{R}^{N \times L}$.
[lit]
$$
    X'_{lm} = \varphi \left( \sum_{i} B_{lmij} X_{ijk} + b_{l} \right)
$$

\begin{figure}\centering
    \includegraphics[width=.5\hsize]{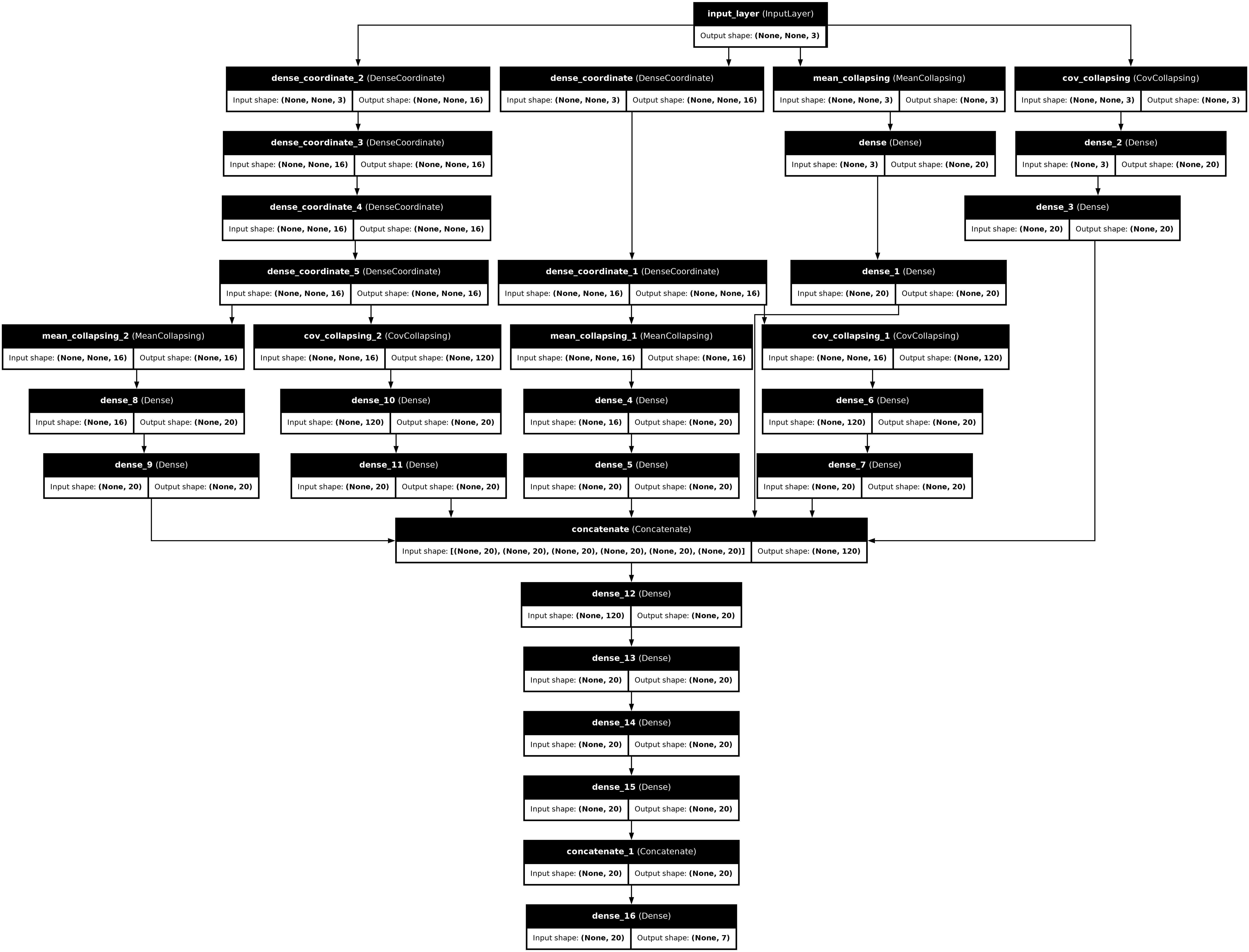}
    \caption{Network structure of a FowardFlow network.}
    \label{fig:network}
\end{figure}

\subsubsection{Rao-Blackwellization}

Rao-Blackwell's theorem (RBT) states that for sufficient statistic $T$ and estimator $δ$ of $θ$,
\begin{theorem}{Blackwell (1947)}
With the notation of the paragraph above, let $T$ be a sufficient statistic for $θ$ and $\hat{θ}$ an estimator of $θ$. Then,
$$
    E((\hat{θ}^*(X) - θ)^2) =  E (\ (E(\hat{θ}(X) | T(X)) - θ)^2\ ) ≤ E( \hat{θ}(X) - θ)^2 ),
$$
where $\hat{θ}^*(X) := E(\hat{θ}(X) | T(X))$
{\it i.e.} for any estimator $θ$, the MSE of the conditional version $ θ | T $ can only have the same or decreased MSE.
\end{theorem}

Informally, for a given value $t = T(X)$, $\hat{θ}$ estimates the same $θ$, so we can as well take the average of all estimates for this $t$. Therefore, RBT can be equivalently defined using a function $f: \mathrm{dom}\,T → Θ$ with $f(t) := E(δ(X) | T(X) = t)$: $E( (f(T ) - θ)^2 ) ≤ E( δ(X) - θ)^2 )$. $f$ can be well approximated by a deep neural network. For example, two full connected layers can define an approximate indicator function on a close interval of $T$ and the bias component of the activation function can represent the expectation. The collapsing layers together with ensuing layers allow to implicitly make use of RBT. The branched structure allows the network to find different types of sufficient statistics for different components of the parameter vector (see appendix).

\subsubsection{Finite Sample Properties}

For data analysis, finite sample properties are of interest. It is well known that consistent estimators can be bias for smaller sample sizes. It is possible to develop bias corrections in these situations. In the current situation, the DNN can automatically learn these corrections when data sets of different sizes are offered during training. This goal also supports branched networks when different branches can potentially specialize on different sample sizes. The sample size dependence is then implicitly learned by the network. In experiments, it was not beneficial to offer the sample size explicitly in a repeated data column (data not shown).

\subsubsection{Hyperparameters}

\begin{table}[h]
\begin{center}
\begin{tabular}{lp{6cm}p{4cm}}
 Name & Description & Range \\
\hline
    $Nbranches$     & Number of branches & 1-3 \\  
    $NdenseBranch$  & Number of dense layers per branch & 0-8\\
    Collapsing      & Collapsing layers applied after each branch (same for all branches) & projection, mean, standard devation (sd), covariance (cov)\\
    $NdensePostColl$ & Number of dense layers after collapsing & 0-4\\
    $NdensePostConcat$ & Number of dense layers after concatenation & 0-4\\
    $NfeaturesBranch$ & Number of features of dense layers in branches & 16-64\\
    $NfeaturesPost$ & Number of features of dense layers in branches & 16-20\\
    $NfeaturesPostConcat$ & Number of features of dense layers after collapsing & 16-20\\
    $Nproj$         & Number of projections performed by the projection collapsing layer & 3\\
    Loss function   & loss function of the network & Mean squared error (mse), Chi-square goodness-of-fit (gof)\\
\end{tabular}
\end{center}
\caption{Definitions of hyperparameters of ForwardFlow networks. Column range indicates values expolored in the simulations.}\label{tab:hyperpars}
\end{table}

Hyper-parameters of the network involve the network structure as well as parameters concerning individual layers. The parameters are defined in table \ref{tab:hyperpars}.

\section{Simulations}

Simulations were conducted for two types of statistical models. The first set of simulations concerns regression models, whereas the second illustrates implicit algorithm estimation for a problem involving a classic EM algorithm in standard modeling. $10^3$ replications were performed per scenario to assess standard errors. Bootstrap confidence intervals were computed using parametric bootstrap resampling with $5 \times 10^3$ replications. In all simulations, during training, sample size was varied uniformly on an interval of sample sizes. On account of the batch based updating of the stochastic gradient descent algorithm, sample size was fixed for each batch. Due to limitations in available compute time, hyper-parameters were not systematically tuned but selected scenarios were evaluated. The full list of evaluated scenarios is given as supplementary information.

\subsection{Aims}

The aim of the simulations is to investigate coverage probabilities of confidence intervals under several simulation scenarios.

\subsubsection{Regression models}

Simulations concerning regression models focus on data contamination. Missingness patterns are induced in the data based on missing at random model. The network needs to perform implicit data imputation for consistent estimation. For the regression model simulations, the impact of discrepancies in sample sizes used during training and testing is investigated.

\subsubsection{Genetic data}

Haplotype frequency estimation also concerns a missing data problem. In this case, consistent estimates cannot be recovered using complete case analysis. Instead, the network has to approximate an EM algorithms implicitly for consistent estimation.

\subsection{Data generation}

\subsubsection{Regression models}

Covariates were simulated with a fixed intercept (value 1) and three covariates which were drawn from a multivariate normal distribution (MVN) with mean 0 and exchangeable covariance matrix (covariance of 0.1). For models with missing data, logit of probability of missingness was simulated according to $x_{1, logit} = (1, x_2)^T \beta_{m1}, x_{2, logit} = (1, x_1, x_3, x_1 x_3)^T \beta_{m2}$ with $β_{m1} = (-1, .5), β_{m2} = (0, .5, .25., .5)$. Missingness indicators were sampled according to these probabilities. Missing values were set to 0 and missingness indicators were added as additional columns to the data.

For linear outcomes, iid standard errors were added to linear predictors to produce the final outcome. For binary outcomes, response values were independently drawn from Bernoulli variables according to probabilities given by the inverse logit of linear predictors.

Outcome columns were added as first column to the covariate matrix without special treatment in the network.

To train, parameters for each simulated data were drawn from an uncorrelated
MVN with mean zero and variance 2. Per batch, sample size was drawn uniformly from the interval $[30, 200]$. 1000 epochs with 100 batches each were run. 300 data sets were drawn per batch, resulting in $3\times 10^7$ data sets generated during training.

Models were evaluated using the mean squared error (MSE) and coverage probabilities of confidence intervals. $10^3$ epochs were used during training for the regression models.

\subsubsection{Genetic data}

A classical problem in human genetics is the estimation of haplotype frequencies (HTFs). Haplotypes are tuples of alleles $h = (h_1, ..., h_K)$ which are realizations of a genetic variant at genetic locations $1, ..., K$ and are taken to be Bernoulli variables for the purpose of this study. A diplotype is a pair of haplotypes $(h_1, h_2)$ representing both parental contributions. In many cases, the diplotype is unobserved. Instead, a genotype is observed which can be defined as the element-wise addition $g = h_1 + h_2$. The genotype can be interpreted as counting the number of one of the alleles per locus which does no longer contain full information on the underlying diplotypes. The haplotype distribution is assumed to be multinomial and, to ensure identifiability, the diplotype distribution is assumed to be the product distribution.

The problem is to estimate HTFs from genotype data. To this end haplotype frequencies were drawn from a Dirichlet distribution. Due to the labeling invariance of haplotypes, the Dirichlet distribution can be chosen to be symmetric, {\it i.e.} being parametrized by a single real number $α > 0$ so that for $\mathrm{Dir}(\boldsymbol{α})$, $\boldsymbol{α} = (α, ... , α)$. From this fixed distribution, HTFs are drawn from which $2 N$ haplotypes are drawn for a sample size of $N$. Haplotypes are converted to binary representation for which each digit corresponds to a locus. The representation allows locus wise counting of alleles producing genotypes. The resulting pair of HTFs and genotypes is used for training.

\subsection{Networks}

\subsubsection{Regression models}

\begin{table}[h]
\begin{center}
\begin{tabular}{ll}
 Parameter & Value \\
\hline
    $NdenseBranch$  & $(0, 2, 4)$\\
    Collapsing      & Projection\\
    $NdensePostColl$ & 0 \\
    $NdensePostConcat$ & 3 \\
    $NfeaturesBranch$ & 32 \\
    $NfeaturesPost$ & -\\
    $NfeaturesPostConcat$ & 32\\
    $Nproj$         & 3\\
    Loss function   & mse\\
\end{tabular}
\end{center}
\caption{Values of hyperparameters used for the simulation of regression models.}\label{tab:parRegr}
\end{table}

Hyperparameters of the network used in the simulation of regression models is given in table \ref{tab:parRegr}. No systematic hyperparameter tuning was performed. The network was gradually expanded until coverage probabilities were deemed sufficient.

\subsubsection{Genetic data}

\begin{table}[h]
\begin{center}
\begin{tabular}{ll}
 Parameter & Value \\
\hline
    $NdenseBranch$  & $(0, 2, 4)$\\
    Collapsing      & Mean, Sdev\\
    $NdensePostColl$ & $(0, 2, 4, 8)$ \\
    $NdensePostConcat$ & 0 \\
    $NfeaturesBranch$ & 16 \\
    $NfeaturesPost$ & 16\\
    Loss function   & mse\\
\end{tabular}
\end{center}
\caption{Values of hyperparameters used for the simulation of regression models.}\label{tab:parGenet}
\end{table}

Hyperparameters of the network used in the simulation of models for genetic data is given in table \ref{tab:parGenet}. Again, n o systematic hyperparameter tuning was performed. The network was gradually expanded until coverage probabilities were deemed sufficient. The main difference to the networks for the regression models is a secondary branching after collapsing.

\section{Results}

\subsection{Regression Models}

\begin{table}[h]\centering
    \begin{tabular}{p{4cm}rrrrrrrr}
Sample Size & Par &25&35&50&75&150&200&300\\
\hline
\multirow{4}{*}{Linear Regression A}
&0&0.976&0.969&0.945&0.951&0.941&0.945&0.942\\
&1&0.964&0.945&0.948&0.955&0.932&0.911&0.897\\
&2&0.982&0.968&0.963&0.952&0.946&0.949&0.942\\
&3&0.955&0.957&0.931&0.966&0.932&0.939&0.929\\
\hline
\multirow{4}{*}{Linear Regression B}
&0&0.96&0.958&0.952&0.928&0.931&0.91&0.878\\
&1&0.949&0.947&0.919&0.927&0.901&0.885&0.842\\
&2&0.977&0.968&0.947&0.936&0.929&0.913&0.905\\
&3&0.962&0.941&0.94&0.944&0.927&0.929&0.927\\
\hline
\multirow{4}{*}{Logistic Regression}
&0&0.862&0.872&0.819&0.795&0.72&0.69&0.632\\
&1&0.896&0.847&0.848&0.828&0.74&0.695&0.657\\
&2&0.888&0.856&0.829&0.812&0.734&0.655&0.623\\
&3&0.874&0.869&0.834&0.838&0.733&0.684&0.63\\
\hline
\end{tabular}

    \caption{Marginal coverage probabilities per parameter for regression models at the 95\%-level. {\em Par}: index of parameter. {\em Sample size}: sample size used in the simulation phase (top row). {\em Linear Regression A}: missing data model trained with 1000 epochs. {\em Linear Regression B}: missing data model trained with 100 epochs. {\em Logistic Regression}: binary output trained with 10 epochs.}\label{tab:resultRegression}
\end{table}

Table \ref{tab:resultRegression} shows marginal coverage probabilities for simulations of the linear model with missing data. Coverage probabilities are nominal for scenario {\em Linear Regression A}, except for the unseen sample size of 300. The unseen sample size 25 shows excess coverage. For scenarios {\em Linear Regression B}, {\em Logistic Regression} under-coverage is observed. In these scenarios, 100 and 10 epochs, respectively, have been used during training.

\subsection{Genetic Data}

\begin{table}[h]\centering
    \begin{tabular}{rrrrr}
&Par&Coverage&rMSE&Bias\\
\hline
&0&0.948&0.010 & 8.84e-05\\
&1&0.940&0.010 & -5.70e-04\\
&2&0.948&0.010 &6.49e-04\\
&3&0.946&0.010 & 6.71e-04\\
&4&0.938&0.010 & -1.04e-04\\
&5&0.944&0.010 &3.93e-04\\
&6&0.936&0.010 &-1.80e-04\\
\hline
\end{tabular}

    \caption{Marginal coverage probabilities per parameter for haplotype frequency estimation at the 95\%-level.}\label{tab:resultHT}
\end{table}

Results for a simulation of HTF estimation are shown in table \ref{tab:resultHT}. Slight under-coverage is observed (average coverage 0.942) for 1000 replications. Estimates are unbiased and the rMSE is 0.01 for all parameters.

\subsection{ABC}

\begin{figure}\centering
    \includegraphics[width=.8\hsize]{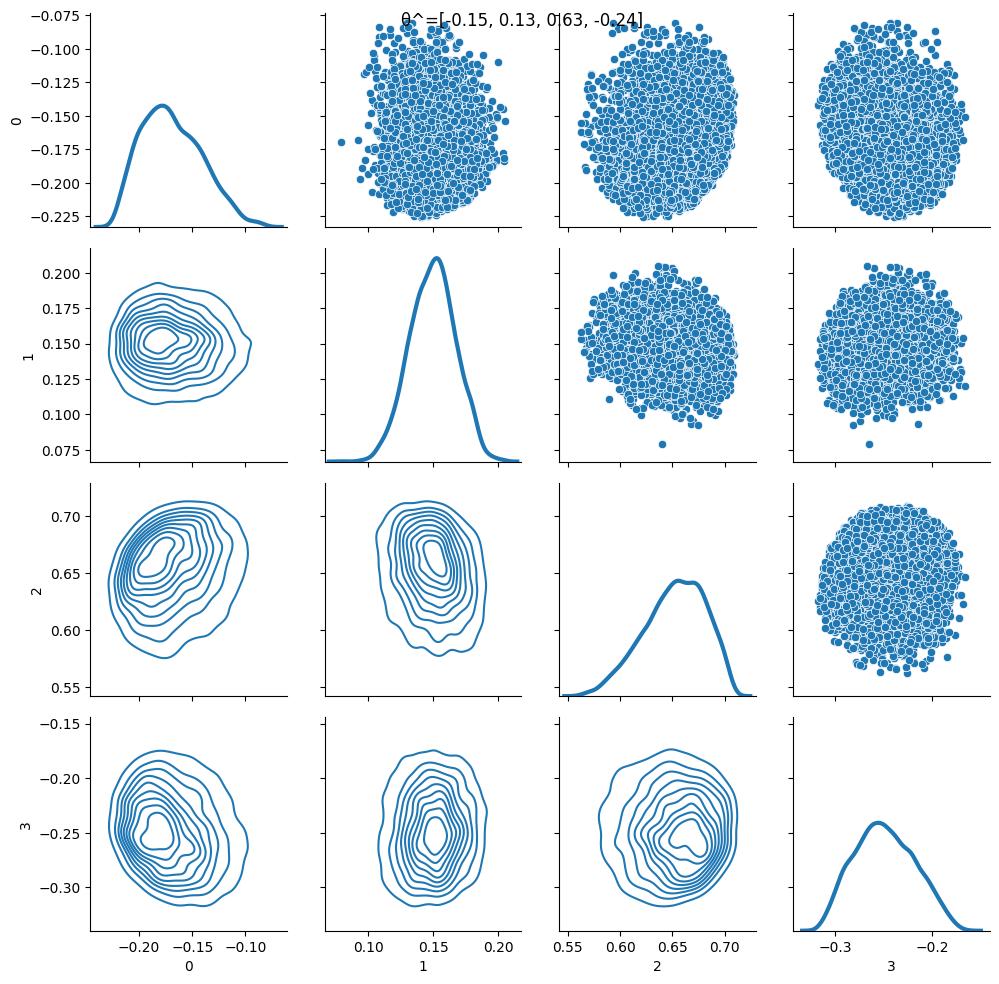}
    \caption{Marginal density estimates (diagonal), pairs-wise controur plots (lower-left triangular matrix) and pair-wise scatter plots (top-right triangular matrix) of the posterior distribution of an example data set derived using ABC. True parameters are shown at the top.}
    \label{fig:abcPost}
\end{figure}

To illustrate the ABC methodology, a single data set was simulated and analyzed. The acceptance rate of this simulation was 5\% and density plots of the estimated pairwise posterior distributions are shown in Figure \ref{fig:abcPost}.

\section{Discussion}

In this conceptual paper, a simulation only, feed-forward network, likelihood-free based approach is proposed for parametric model inference. Simulations show that the approach can deal with data contamination, can implicitly learn algorithms, and can be employed in Bayesian inference. Data simulation is often simpler to implement than implementing the data likelihood and is required for model validation in any case. The promise of having completed the full work of model implementation is attractive in many applications. Even if the likelihood is straightforward to implement it may costly to evaluate which complicates model implementation, a step that can be omitted in the {\it ForwardFlow} framework. Our simulation results suggest exact coverage of marginal confidence intervals implying UMVU properties of the trained network.

The ideas in this paper are borrowed from Bayesian models which use a more complicated network structure as compared to {\it ForwardFlow}. Apart from a {\it SummaryNet} which is comparable to the full {\it ForwardFlow} network, a normalizing flow is learned that conditionally on summaries from the {\it SummaryNet} learns the parameter distribution. In a second step, the flow can be used generatively to draw samples from the posterior distribution. In its current form, {\it ForwardFlow} has to resort to a (non-)parametric bootstrap that has to generate data sets in addition to employing the network. The trade-off is between using a simpler network and generating new data ({\it ForwardFLow}) and using only a more complex network ({\it BayesFlow}). For both approaches, it is possible to learn secondary models that learn, say, the multivariate cumulants of the target distribution (confidence, posterior). The models would have to be trained on accurate approximations of the target distributions, {\it i.e.} starting from a simulated parameter, many samples would have to be sampled from the confidence or posterior distribution the multivariate cumulants of which are then learn by the network.

A standard approach to finite sample exactness is based on cumulative probabilities and is often computationally intensive \cite{martin_plausibility_2015, bohringer_exact_2022}. Deep learning models can acquire this exactness automatically when trained with varying sample sizes. Our simulations confirm this behavior but, surprisingly, it seems important to train the model on all samples sizes to be later analyzed. Under-coverage was observed in some simulations when the analyzed sample size was larger than seen in training data. This behavior might be caused by the relatively low sample sizes considered here, but is an impaortant caveat when employing the networks.

In the simulations only marginal confidence intervals were analyzed. Multivariate confidence intervals can be constructed using concepts from data depth \cite{ignacio_data_2009}. In case a secondary model has been trained to approximate the confidence distribution, multivariate cumulants can be used to construct joint confidence sets \cite{mccullagh_cumulants_2009}.

With respect to the analysis of data contamination, only a fixed missingness mechanism was simulated. Results are therefore purely conceptual as in practice, an unknown missingness mechanism has to be allowed for. For training a useful network, random missingness inducing models would therefore have to be drawn and employed on a data set that was drawn in parallel. Such a missingness model has been implemented in the software but has not yet been extensively simulated. The network would be likely to benefit from switching from fully connected layers to attention based layers which can help to accommodate the symmetry of columns of tabular data \cite{vaswani_attention_2017}.

In the simulations, we demonstrate examples where, in the genetic situation, the development of an EM-algorithm is obviated by the simulation only approach. We give a compact implementation in the appendix for both data simulation and estimation. The difference is roughly a factor of 10 in lines of code, indicating that reduction in development time is an additional advantage of {\it ForwardFlow}.

The choice of network structure has to be motivated by conceptual considerations and makes idealized assumptions. Due to the limited numeric accuracy of standard implementations, it is unclear in how far these concepts are actually reflected in the network. This quantization of the numeric parameter representation is a deliberate design principle of deep neural networks, trading off accuracy against a larger number of parallelizable computations. Here, we surmise that accuracy can be regained by combining results from several computation paths. However, the final choice of network depends on hyper-parameter tuning and might result in uniintuitive networks.\par

In summary, we believe that the {\it ForwardFlow} approach can potentially offer significant advantages in practical data analysis by combining fast model implementation, robust inference and finite sample exactness. More work is needed to offer pre-trained models that can analyze a wide class of parametric models.


\newpage

\printbibliography

\newpage
\appendix
\section{Proofs}

\subsection{Data conatination (lemma \ref{lemma:contamination})}

\begin{proof}
The proof is based on the fact that a proper estimate can be obtained from the known, uncontaminated part of the data.

Without loss of generality, assume a fixed $n$ and $k$ to be the number of uncontaminated samples $1 ≤ i ≤ k$.
Denote with
    $X_{u} = (X_{ni})_{i = 1}^k$
the uncontaminated part of the data, and with
    $X_{c} = (X_{ni})_{i = k+1}^n$
the contaminated part. Then
    $E(\hat{θ}(X_u)) = θ$.
By assumption, $E(\hat{θ}(X_c)) = θ + b(θ) = θ_c$.
Using $ E(\hat{θ}(X_u) - \hat{θ}(X_c)) = θ - θ_c = b(θ)$ defines a bias function and the construction allows to disentangle $θ$ and $b(θ)$.
\end{proof}

\subsection{Importance sampling for ABC (lemma \ref{lemma:importance})}

\begin{proof}
Let $\hat \theta^D \in \mathbb{R}^k$ be the MLE of the data. Let $F$ follow a multivariate normal distribution such as $F \sim \mathcal{N}(\hat \theta^D, I_k \epsilon)$, where $I_k$ is the identity matrix and $\epsilon > 0$. 

To approximate the posterior distribution, parameters are drawn from the the mixture distribution
    $Z = \sum_i^N \frac{1}{N} Z_i$, where
    $Z_i \mathrm{iid} \sim N(\theta_i, \Sigma)$.
The covariance matrix is a scaled version of the empirical covariance matrix of a sample of already drawn posterior samples. We here assume that this covariance matrix is known.

It is to be shown that acceptance probability increases, {\it i.e.}
$$
          P(d(\hat \theta^*, \hat \theta) < ε)
        ≤ P(d(\hat \theta^{'}, \hat \theta) < ε),
$$
for parameters $θ'$ drawn from the new prior compared to samples $θ^*$ from the old prior. Note, that
    $E(d(\hat \theta^{'}, \hat \theta)) = h(s)$
is a monotonous function of scaling factor $s$, {\it i.e.} $E(d(\hat \theta^{'}, \hat \theta)) → 0$ for $h(s) → 0$ and consequently $p' = P(d(\hat \theta^{'}, \hat \theta) < ε) → 1$. This implies that for any
    $p = P(d(\hat \theta^*, \hat \theta) < ε)$,
an $s$ can be chosen, so that $p' ≥ p$.
\end{proof}

\section{Haplotype simulation \& EM algorithm}

\subsection{Simulation of genotypes}

The following R-code allows to simulate genotype data. Function \texttt{simulateGenotypes} takes arguments {\it N} (sample size), {\it Nloci} (number of loci to simulate), and {\it hapDIst} (vector of HT frequencies). The code assumes the function \texttt{int2bin} that converts an integer to its binary representation. In total, 10 lines of code are required. This simulation assumes bi-allelic loci.

\begin{lstlisting}[language=R]
simulateDiplotypes = function(N = 1, hapDist, names = 0:(length(hapDist) - 1))
  matrix(t(rmultinom(2 * N, 1, hapDist)) %*% t(t(names)), ncol = 2)
dt2alleles = function(dt, Nloci)
  sapply(dt, int2bin, digits = Nloci)
dt2gt = function(dt, Nloci)
  apply(dt2alleles(dt, Nloci), 2, sum)
dts2gts = function(dts, Nloci)
  t(apply(dts, 1, dt2gt, Nloci = Nloci))
simulateGenotypes = function(N, Nloci, hapDist)
  dts2gts(simulateDiplotypes(N, hapDist), Nloci)
\end{lstlisting}

\subsection{EM algorithm}

The EM-algorithm takes substantially more code as diplotypes (pairs of haplotypes) compatible with genotypes have to be reconstructed.

The main loop of the EM algorithm iterates E- and M-steps. {\it gts} is the matrix of genotypes, {\it Nitmax} is the maximal number of iterations and {\it eps} is the convergence tolerance (some code omitted).

\begin{lstlisting}[language=R]
htFreqEstEM = function(gts, eps = 1e-5, Nitmax = 1e2) {
  # ... function setup (10 lines)
  for (i in 1:Nitmax) {
    # E-step
    dtfs = htfs2dtfs(htfs, dtfI);
    dtCont = colStd(t(dtsI) * dtfs); # E-step
    # M-step
    htfsN = vn(apply(t(dtCont) %*% Mdt2ht, 2, sum));
    if (max(abs(htfsN - htfs)) <= eps) break;
    htfs = htfsN;
  }
  return(list(htfs = htfsN, converged = i < Nitmax, eps = eps, iterations = i));	
}
\end{lstlisting}

The reconstruction code {\it htfs2dtfs} requires 10 helper functions and about 100 lines of code. Helper function $gts2dts$ (compute all diplotypes compatible with given genotypes as explicit pairs of haplotypes), $dtsapply$ (iterate results from $gts2dts$), $Npairs$ (number of unordered pairs for $N$ objects), $merge.multi.list$ (compute cartesian product of lists), and $array.extract$ (extract elements from array by index tuples) have been omitted.

\begin{lstlisting}[language=R]
# diplotype reconstruction as indicator matrix of possible diplotypes
dtsReconstruct = function(gts) {
	Nloci = ncol(gts);
	templ = rep(0, Npairs(2^Nloci));
	dtsR = gts2dts(snpGtsSplit(gts));
	dtsEnc = dtsapply(dtsR, gtFromPair0, Nalleles = 7);
	dtsI = t(sapply(dtsEnc, vector.assign, e = 1, v = templ));
	return(dtsI);
}
# to which haplotypes does a given diplotype contribute?
dtContribMat = function(Nloci) {
	Nhts = 2^Nloci;
	Ndts = Npairs(Nhts);
	templ = rep(0, 2^Nloci);
	dt2ht = apply(sapply(1:Ndts, pairFromGt, Nalleles = Nhts), 2, vector.assign, e = 1, v = templ);
	dt2ht2 = t(dt2ht) * 2 / apply(dt2ht, 2, sum);	# normalized to two alleles
	return(dt2ht2);
}
# map diplotype indices to indeces of Nhts x Nhts diplotype frequency matrix
dtFreqExtraction = function(Nhts) {
	combs = merge.multi.list(list(list(1:Nhts), list(1:Nhts)));
	combs = combs[combs[, 1] <= combs[, 2], ];
	gts = apply(combs, 1, gtFromPair, Nalleles = Nhts);	
	return(combs[order(gts), ]);
}
# express diplotype frequncies in a Nhts x Nhts matrix, assuming HWE
htfs2dtfs = function(htfs, dtfI) {
	dtfsRaw = htfs %*% t(htfs);
	dtfs = array.extract(2*dtfsRaw - diag(diag(dtfsRaw)), dtfI[, 1], dtfI[, 2]);
	return(dtfs);
}
# standardize matrix by column
colStd = function(m)t(t(m) / apply(m, 2, sum))
rowStd = function(m)t(colStd(t(m)))
\end{lstlisting}




\end{document}